\documentclass{amsart}
\usepackage{colonequals,enumitem,amssymb,amsmath}
\usepackage{graphicx,caption,subcaption}
\usepackage{multirow}

\theoremstyle{plain}
\newtheorem{thm}{Theorem}[section]
\newtheorem{lem}[thm]{Lemma}
\newtheorem{prop}[thm]{Proposition}
\newtheorem{conj}[thm]{Conjecture}

\newcommand{\ZZ}{\mathbb{Z}}

\DeclareMathOperator{\core}{\mathfrak{core}}

\title{Small Cores in 3-uniform Hypergraphs}

\begin{document}

\begin{abstract}
The main result of this paper is that for any $c>0$ and for large enough $n$ if the number of edges in a 3-uniform hypergraph is at least $cn^2$ then there is a core (subgraph with minimum degree at least 2) on at most 15 vertices. We conjecture that our result is not sharp and 15 can be replaced by 9. Such an improvement seems to be out of reach, since it would imply the following case of a long-standing conjecture by Brown, Erd\H os, and S\'os; if there is no set of 9 vertices that span at least 6 edges of a 3-uniform hypergraph then it is sparse. 
\end{abstract}

\author{David Solymosi}
\address{Department of Computer Science, University of Toronto 
10 Kings College Road, Toronto, Ontario M5S 3G4, Canada
}
\email{solymosi@cs.toronto.edu}
\author{Jozsef Solymosi}
\address{Department of Mathematics, University of British Columbia 
1984 Mathematics Road, Vancouver, 
British Columbia V6T 1Z2, Canada}
\email{solymosi@math.ubc.ca}

\thanks{Research was supported by NSERC and the second author partially supported by ERC Advanced Research Grant
AdG. 321104, and by Hungarian National Research Grant NK 104183}

\maketitle

\section{Introduction}
The subject of this paper is finding small cores in 3-uniform hypergraphs in terms of the number of edges.
A core is a non-empty subgraph in which every vertex has degree at least two. We introduce the notation
\[ \core(n,k) = \min \{ t : e(H_n^3) \geq t \Rightarrow \mbox{$H_n^3$ contains a core on at most $k$ vertices} \}, \]
where $H_n^3$ is a $3$-uniform hypergraph on $n$ vertices, and $e(H_n^3)$ is the number of edges
of $H_n^3$. 

Cores are important objects in the theory of hypergraphs. One direction of research is on finding sharp thresholds for the appearance of cores in random hypergraphs, as in \cite{Mike}, \cite{JL}, and in \cite{PSW}. Analyzing cores in hypergraphs turned out to be useful in number theory;  in \cite{CGPT} and \cite{Ab} the authors translate some aspects of Pomerance's problem on prime factorization \cite{Po} to hypergraphs where the existence of cores is the key property. 
More surprisingly cores of hypergraphs appeared in models of protein interaction networks \cite{BH,RTP}. Algorithmic problems of finding minimum cores were considered in \cite{Da}.  Minimum cores in terms of the number of edges might be seen as a generalization of the girth problem in graphs, as the girth of a graph is the size of the smallest subgraph with minimum degree 2. For references about the girth problem and for some variations of the possible notation of girth of hypergraphs we refer to \cite{FS} Section 4, and \cite{EL} Section 4. Since the girth problem, Erd\H os' girth conjecture, is open for graphs one can expect determining $\core(n,k)$ to be hard. Indeed, most of our results are not sharp and possible improvements would imply or require progress in well know conjectures.

Our main result is closely related to the following conjecture:
\begin{conj}[Brown, Erd\H os, and S\'os]\label{BES}
For every $\ell \geq 3$ and $c>0$ there exists an $n_0 = n_0(c)$ such that if $n>n_0$ and $e(H_n^3) \geq cn^2$
then there exists a $F_{\ell+3}^3 \subseteq H_n^3$ such that $e(F_{\ell+3}^3) \geq \ell$.
\end{conj}
Using the notation of \cite{EL} the smallest $\ell$ such that there is an $F_{\ell+3}^3 \subseteq H_n^3$ is the {\em $(-3)$-girth of $H_n^3.$}
Sometimes we are also going to use the ``$(\ell+3,\ell)$ {\em conjecture"} notation for a given integer $\ell$.

To this date the best bound in this direction is the result of S\'ark\"ozy and Selkow \cite{SaSe}.
\begin{thm}[S\'ark\"ozy and Selkow]\label{SS}
For every $\ell \geq 3$ and $c>0$ there exists an $n_0 = n_0(c)$ such that if $n>n_0$ and $e(H_n^3) \geq cn^2$
then there exists a $F_{\ell+2+\lfloor \log_2{\ell}\rfloor}^3 \subseteq H_n^3$ such that $e(F_{\ell+2+\lfloor \log_2{\ell}\rfloor}^3) \geq \ell$.
\end{thm}
Their proof relies on Szemer\'edi's regularity lemma. In our application we are going to improve their result for the $\ell=10$ case using a hypergraph regularity lemma of Frankl and R\"odl \cite{FR}. This result might be of independent interest as this is the first improvement in the last decade in this important problem.

The $\ell=3$ case was proved by Ruzsa and Szemer\'edi \cite{RSz}. We are going to use the following
quantitative version:
\begin{thm}\label{63density}
For every $c>0$ there exists a $\delta>0$ such that if the number of edges is at least $cn^2$
then there exists $\delta n^3$ subgraphs $F_6^3 \subseteq H_n^3$ such that $e(F_6^3) \geq 3$.
\end{thm}

An important part of the paper deals with the range when the size of the hypergraph is $o(n^2),$ where the relation
to the Brown, Erd\H os, and S\'os conjecture is clear, however, finding sharp bounds on $\core$ is an
interesting open problem for any edge densities. We collect bounds which we were able to obtain for various ranges. These are summarized in Table \ref{table:mind}.

\section{Very small $k$ (4 to 8)}

\subsection{$k=4$}

The smallest possible core in a 3-uniform hypergraph is the hypergraph $K_4^3-e,$ a clique with one edge removed. (Or equivalently, the unique 3-uniform hypergraph on 4 vertices with 3 edges.)
The presence of such a subgraph belongs to the family of classical Tur\'an-type hypergraph problems. The exact Tur\'an density is not known. (The Tur\'an density $\pi(H)$ of a family $H$ of $k$-graphs is the limit as $n$ tends to infinity of the maximum edge density of an $H$-free $k$-graph on $n$ vertices.) The best lower bound density is due to Frankl and F\"uredi \cite{FF}, it is $\pi(H)\geq 2/7=0.2857\ldots$ while the upper bound, $0.2871$ is due to Baber and Talbot \cite{BT}. 

\medskip
In the problem of finding larger ($k\geq 5$) cores we are interested about the magnitude of the bounds, not the constant multipliers. Therefore, we can suppose w.l.o.g. that our 3-uniform hypergraphs are tripartite.  It makes our analysis simpler a bit. (A random partition would leave an edge in the tripartite graph with probability 2/9.) 

\subsection{$k=5$}

\medskip
A hypergraph on five vertices with at least four edges will contain a core of size 5 or 4. In a tripartite hypergraph 4 vertices span at most two edges, so a core has size at least 5. The upper and a lower bounds on $\core(n,5)$ follow from earlier results. 

\begin{thm}\label{core5}
There are constants $c_1>0$ and $c_2>0$ such that
\[ c_1n^{5/2}\leq \core(n,5)\leq c_2n^{5/2}.\]
\end{thm}

\begin{figure}[ht]
        \centering
        \begin{subfigure}[b]{0.2\textwidth}
                \includegraphics[width=\textwidth]{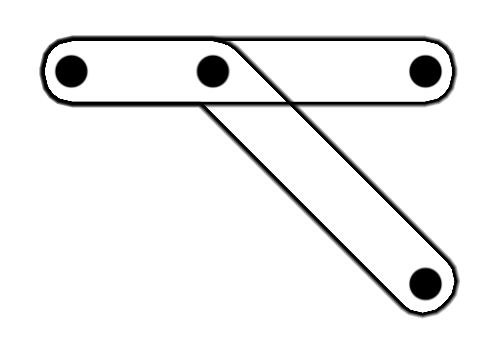}
                \caption{(a)}
                \label{fig:nonlin}
        \end{subfigure}%
        ~ 
        \begin{subfigure}[b]{0.2\textwidth}
                \includegraphics[width=\textwidth]{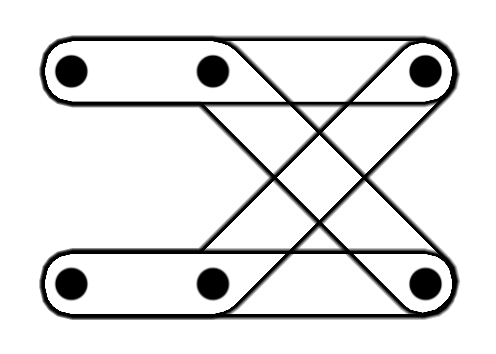}
                \caption{(b)}
                \label{fig:6core}
        \end{subfigure}
        ~ 
        \begin{subfigure}[b]{0.2\textwidth}
                \includegraphics[width=\textwidth]{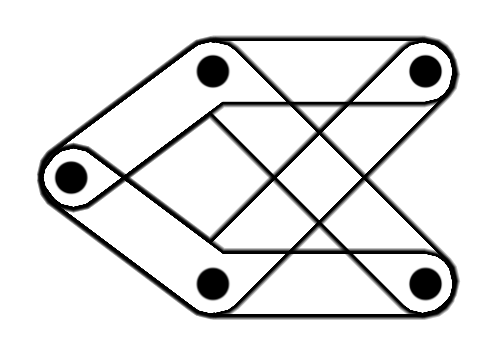}
                \caption{(c)}
                \label{fig:5core}
        \end{subfigure}
        \caption{}
\end{figure}

When $e(H_n^3) > n^{5/2}$, we show that a core of size at most $5$ must exist. This follows from the existence of a $K^3_{(2,2,1)}$, the complete tripartite graph on 5 vertices, which is a core. The extremal problem for complete tripartite graphs  was considered by Erd\H os in \cite{Er}. For the sake of completeness we present the simple calculation showing the upper bound here as we will use similar arguments later. 
We count the number of edge pairs that intersect in two points (Figure~\ref{fig:nonlin}), which is
\[ \sum_{\substack{v_i \neq v_j \in H_n^3\\i<j}} \binom{\deg(v_i,v_j)}{2}, \]
where $\deg(v_i, v_j)$ is the number of edges that contain both $v_i$ and $v_j$. The number of edges, $e,$ is large enough so that
\[ \sum_{\substack{v_i \neq v_j \in H_n^3\\i<j}} \binom{\deg(v_i,v_j)}{2} \geq \binom{n}{2} \binom{e/\binom{n}{2}}{2} \approx \frac{e^2}{n^2} > {n\choose 3}. \]
The inequality shows that at least two
intersecting edge pairs intersect in three points, as seen in Figure~\ref{fig:5core}.
This subgraph containing four edges
on five points is a core, so $\core(n,5) \leq n^{5/2}$.

\medskip
The more interesting lower bound follows from a nice result of Mubayi who proved in that the maximum number of edges in an $r$-uniform graph on $n$ vertices without a complete $r$-graph $K_{(1,\ldots, 1, 2, t+1)}$ is $\frac{\sqrt{t}}{r!}n^{r-1/2}.$ (For details see Theorem~3.1 in \cite{Mu}.) Mubayi's construction builds on some ideas from F\"uredi's paper \cite{Fu2}.

\medskip

\subsection{$k=6, 7, 8$}
\medskip
The next three cases have the same, quadratic upper and lower bounds.

\begin{thm}\label{core678}
There are constants $c_1>0$ and $c_2>0$ such that
\[ c_1n^2\leq \core(n,8)\leq \core(n,7)\leq \core(n,6)\leq c_2n^{2}.\]
\end{thm}

The upper bound follows from a calculation similar to the $k=5$ case above. 
\begin{lem}
\[ \core(n,6) < n^{2}.\]
\end{lem}

\begin{proof}
If the number of edges, $e,$ is large enough so that
\[ \sum_{\substack{v_i \neq v_j \in H_n^3\\i<j}} \binom{\deg(v_i,v_j)}{2} \geq \binom{n}{2} \binom{e/\binom{n}{2}}{2} \approx \frac{e^2}{n^2} > {n\choose 2}, \]
then at least two
intersecting edge pairs intersect in the two degree one vertices, as seen
in Figure~\ref{fig:6core}. This subgraph containing four edges
on at most 6 vertices is a core, so $\core(n,6)< n^{2}$.
\end{proof}

To prove the lower bound in Theorem \ref{core678} we give an algebraic construction for $H_{3n}^3$ with $e(H_{3n}^3) = n^2$ which does not contain a core of size 8 or less.

\begin{lem}
There is a constant $c>0$  such that
\[ \core(n,8)\geq cn^{2}.\]
\end{lem}

\begin{proof}
Let $A, B, C \cong \ZZ / p\ZZ$ for some large prime $p$, and consider
the 3-partite 3-uniform hypergraph over vertices $A\cup B\cup C$
which has edges
$$\left\lbrace \{ a, b, c\} \big| a\in A, b\in B, c\in C, a + b = c \right\rbrace.$$

Suppose $U \subseteq A, V \subseteq B, W \subseteq C$ are nonempty
subsets such that the induced subgraph on $U\cup V\cup W$ is a core.
We will show that $|U|+|V|+|W| \geq 9$.

Note that we require at least $2\cdot\max(|U|,|V|,|W|)$ edges for
every vertex to have degree two, but we can only have at most
$\min(|U||V|, |U||W|, |V||W|)$ edges, as any two vertices uniquely
determine an edge. For $|U|+|V|+|W| < 9$ this leaves us with two
cases, when they all have size 2, or when two of them have
size 3 and one has size 2.

Let $U = \{u_1, \ldots u_{|U|}\}$, $V = \{v_1, \ldots v_{|V|}\}$,
and $W = \{w_1, \ldots w_{|W|}\}$.

If $|U|=|V|=|W|=2$, we must have $4$ edges, so every $\{ u_i, v_j \}$
will be contained in an edge.
The expressions $u_1 + v_1, u_1+v_2, u_2+v_1, u_2+v_2$
give two distinct values, so $u_1+v_1 = u_2 + v_2$ and $u_1+v_2 = u_2 + v_1$.
Combining these expressions, $u_2 + v_1 - v_2 + v_1 = u_2 + v_2$, so
$(v_1 - v_2) + (v_1 - v_2) = 0$, therefore $v_1 = v_2$ if $2 \not| p$.

If $|U|=|W|=3$ and $|V|=2$, we must have $6$ edges, so every $\{ u_i, v_j \}$
will be contained in an edge.
The expressions
$u_1 + v_1, u_1 + v_2, u_2 + v_1, u_2 + v_2, u_3 + v_1, u_3 + v_2$
only give three distinct values. No three of these expressions can
be equal, since then $u_1$, $u_2$, and $u_3$ are not all distinct.
We may assume
\begin{equation}\label{c8-1}
u_1 + v_1 = u_2 + v_2,
\end{equation}
then $u_1+v_2 \neq u_2 + v_1$ by the $|U|=|V|=|W|=2$ case above.\\
The other two distinct values must correspond to
\begin{equation}\label{c8-2}
u_1+v_2 = u_3 + v_1 \mbox{, and}
\end{equation}
\begin{equation}\label{c8-3}
u_2+v_1=u_3+v_2.
\end{equation}
Now consider \eqref{c8-1}-\eqref{c8-2}+\eqref{c8-1}, that is,
\begin{align*}
v_1 - v_2 + u_2 + v_1 &= u_2 + v_2 - u_3 - v_1 + u_3 + v_2,\\
3 v_1 &= 3 v_2,
\end{align*}
so $v_1 = v_2$ if $3 \not| p$.

The case $|V|=|W|=3$ and $|U|=2$ is the same as above, and if
$|U|=|V|=3$ and $|W|=2$ then we can repeat the argument with
expressions of the form $u_i + (-w_i)$.
\end{proof}

\section{Small $k$ (9 to 15)}

Our main result in this section is to prove that if $e(H_n^3) = \Omega(n^2)$, then $H_n^3$ contains a core of size
at most 15. We conjecture that our result is not sharp and 15 can be replaced by 9. As before, we will suppose w.l.o.g.~that $H_n^3$ is tripartite, with disjoint vertex sets $V_1, V_2,$ and $V_3$ so that every edge has one vertex in each class. 

\subsection{$k=9$}

\medskip
We conjecture that that if $e(H_n^3) = \Omega(n^2),$ then $H_n^3$ contains a core of size at most 9. Such improvement seems to be out of reach, since it would imply the $\ell=6$ case of the Brown, Erd\H os, S\'os conjecture (Conjecture \ref{BES}). In a tripartite core on 9 vertices at least 3 vertices are in the same vertex class, so the number of edges is at least 6 in the core since every vertex has degree at least two. This would imply the $(9,6)$ conjecture above.

\subsection{$k=10$}

\begin{prop}\label{core10}
There is a $c>0$ such that there are graphs $H_n^3$ without a core of size 10 with $e({H_n^3})=cn^2$ for arbitrary large $n.$
\end{prop}
If a core has 10 vertices then at least 4 of the vertices are in the same vertex class so the core has at least 8 edges since every vertex has degree at least two. So, we had 8 edges on 10 vertices. An elegant probabilistic construction shows that such configurations can be avoided in dense linear 3-uniform hypergraphs;  Brown, Erd\H os, and S\'os showed in  \cite{BES} that their conjecture (if true) is sharp. For every $\ell$ there is a $c>0$ such that there are arbitrary large 3-uniform hypergraphs on $n$ vertices having at least $cn^2$ edges, without $\ell+2$ vertices spanning at least $\ell$ edges.

\subsection{$k=11$}

\medskip
A core on 11 vertices would also contain at least 8 edges since at least four vertices will be in one of the vertex classes and in a core every vertex has degree at least two. Proving that $\core(n,11)=o(n^2)$ would again imply the Brown, Erd\H os, S\'os conjecture for the $\ell=8$ case. (The $(11,8)$ conjecture.)

\subsection{$k=12$}

\medskip
In the range of core $k\geq 12$ and  edge density $O(n^{3/2})\leq e(H_n^3)$ we can suppose w.l.o.g. that
$H_n^3$ is a {\em linear} tripartite hypergraph, i.e. no edges intersect in two vertices. More precisely we show that 

\begin{prop}\label{linear}
For every $\varepsilon >0$ there is a threshold $n_0$ such that if the number of edge-pairs intersecting in 2 vertices is at least $(1/2+\varepsilon)n^{3/2}$ in $H_n^3,$ $n\geq n_0,$ then it contains  core of size at most 12.
\end{prop}

\begin{proof}
Let us define an auxiliary graph $G_n$ where two vertices are connected by an edge iff they are the single vertices of two edges intersecting in two vertices in $H_n^3.$ This is a simple graph since a double edge here would determine a core on at most 6 vertices in $H_n^3.$ If graph $G_n$ contains a $C_4$ then the four defining edge-pairs in $H_n^3$ would form a core on at most 12 vertices. $C_4$-free graphs contain at most $\sim n^{3/2}/2$ edges. (see e.g. in F\"uredi \cite{Furedi})
\end{proof}

This shows that either there is a core on at most 12 vertices, or we can remove $O(n^{3/2})$ edges from $H_n^3$ (one from each intersecting edge-pair) to make it linear. We will assume
that $H_n^3$ is a tripartite and linear hypergraph.

\medskip
We can't prove $\core(12)=o(n^2)$ in general, however we can prove it for an important family of linear triple systems. The first open case of Conjecture \ref{BES} is $\ell=4.$ The conjecture has been proved for triple systems where additional arithmetic structures can be find, like in Szemer\'edi's theorem on $k$-term arithmetic progressions and in its generalizations. For such triple systems a general theorem was proved by the second author in \cite{SS}. For every $c>0$ there is a threshold $m_0$ such that if $(\Gamma_m,{\cdot})$ is a group with $m\geq m_0$ elements, then no matter how we choose $cm^2$ triples of the form $(a,b,a{\cdot} b),$ where $a,b\in \Gamma_m,$ there will be 7 elements of the group carrying at least four of the selected triples. If $S\subset \Gamma_m\times \Gamma_m$ and the density of $S$ is denoted by $\kappa,$ (i.e. $\kappa=|S|/m^2$) then the following holds.

\begin{thm}[Theorem 6 in \cite{SS}] There is a constant $\mu > 0$ depending on $\kappa$  only so that the number of $(\alpha,\beta),$ $(\alpha,\gamma),$ $(\delta,\gamma),$ $(\delta,\beta)$ quadruples from $S$ such that $\alpha\beta=\delta\gamma$ 
 in the group is at least 
$$\mu  e^{\nu\sqrt{\log{m}}}m^2.$$
\end{thm}

It gives a lower bound on the number of  $(7,4)$ configurations; every triple $(a,b,ab)$ is uniquely determined by $(a,b)\in   \Gamma_m\times \Gamma_m.$ The triples $(a,b,ab), (a,c,ac),(d,c,dc),$ and $(d,b,db)$ determine at most seven elements of $\Gamma_m$ which are $a,b,c,d, ab=dc, ac,$ and $db.$ See fig \ref{Group12}.

\begin{figure}[h]
\centering
\includegraphics[width=0.3\linewidth]{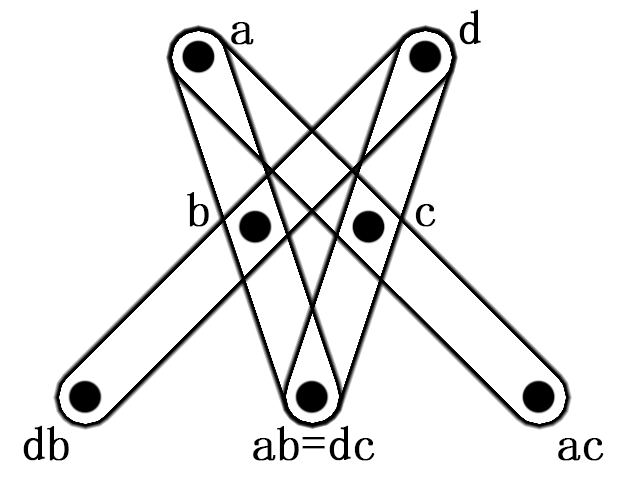}
\caption{}
\label{Group12}
\end{figure}

The last two elements might coincide in which case we are done because it would form a core. We can suppose that the two elements, $ac$ and $db$ are distinct in every $(7,4)$ configurations. If $m$ is large enough then 

$$\mu  e^{\nu\sqrt{\log{m}}}m^2\gg m^2$$
which guarantees that there will be at least three $(7,4)$ configurations $(a_i,b_i,a_ib_i), $ $(a_i,c_i,ac),$ $(d_i,c_i,d_ic_i),$ $(d_i,b_i,db)$ 
such that $a_ib_i=d_ic_i,$  $3\leq i\leq 1,$ (since they were $(7,4)$ configurations) moreover $a_ic_i=a_jc_j=ac$ and $d_ib_i=d_jb_j=db,$ $3\leq i,j\leq 1,$ because there are less then $m^2$ distinct $ac, db$ pairs.
We show that the union of three such $(7,4)$ configurations contains a core on 12 or less vertices. If two out of the three are edge-disjoint then it is a core because they share the only two degree one vertices of the configuration. The two common elements $ac$ and $db$ with the edge  $(a,b,ab)$ determines the $(7,4)$ configuration; $b$ and $db$ gives $d,$ i.e. $db\cdot b^{-1}=d.$ Note that $ab=dc,$ so $d$ and $ab$ gives $c,$ $d^{-1}\cdot(dc)=c.$ Finally we get $a$ by $(ac)\cdot c^{-1}=a.$ So, the $(7,4)$ configurations cannot share such edge and the same holds to edge $(d,c,dc),$ by symmetry. The only remaining case (up to symmetry) is when the three configurations share the $(a,c,ac)$ edge and the other edges are distinct. The union of them has degree one in $ac,$ degree three in $a$ and $c$ and degree at least two everywhere else. If we remove vertex $ac$ and the common edge  $(a,c,ac)$ then the remaining triples determine a core on at most 12 vertices (some vertices might coincide).

\subsection{$k=13$}

\medskip
A core on 13 vertices has at least 5 vertices in one vertex class therefore it has at least 10 edges. It is a $(13,10)$ configuration, so a proof of $\core{(n,13)}=o(n^2)$ would imply Conjecture \ref{BES} for $\ell=10.$ We can only prove the existence of a $(14,10)$ configuration.

\subsection{$k=14$}
\bigskip
A core on 14 vertices contains at least 10 edges (since $\lceil 14\cdot 2 /3 \rceil = 10$).
We were unable to prove that $\core(n,14)$ is $o(n^2)$, however we improve the
previously best known bound for the Brown Erd\H os S\'os conjecture by Selkow and S\'ark\"ozy.
We prove that if $H_n^3$ contains no 14 vertices spanning at least 10 edges then
$e(H_n^3) = o(n^2)$. (The bound of Selkow and S\'ark\"ozy, Theorem \ref{SS}, guarantees 9 edges only on 14 vertices.) 

\begin{thm}\label{core14}
If $H_n^3$ contains no subgraph $F_{14}^3$ such that $e(F_{14}^3) = 10$, then $e(H_n^3) = o(n^2)$.
\end{thm}

Our main tool in proving this theorem is a generalization of Theorem~\ref{63density} by
Frankl and R\"odl \cite{FR}. 

\begin{thm}[Removal Lemma for 3-uniform Hypergraphs]\label{rem}
For every $c>0$ there exists a $c'>0$ such that if $H_n^3$ contains at least
$cn^3$ pairwise edge-disjoint cliques $K_4^3$, then it contains $c'n^4$ cliques $K_4^3$.
\end{thm}

\begin{figure}[h]
\centering
\includegraphics[width=0.6\linewidth]{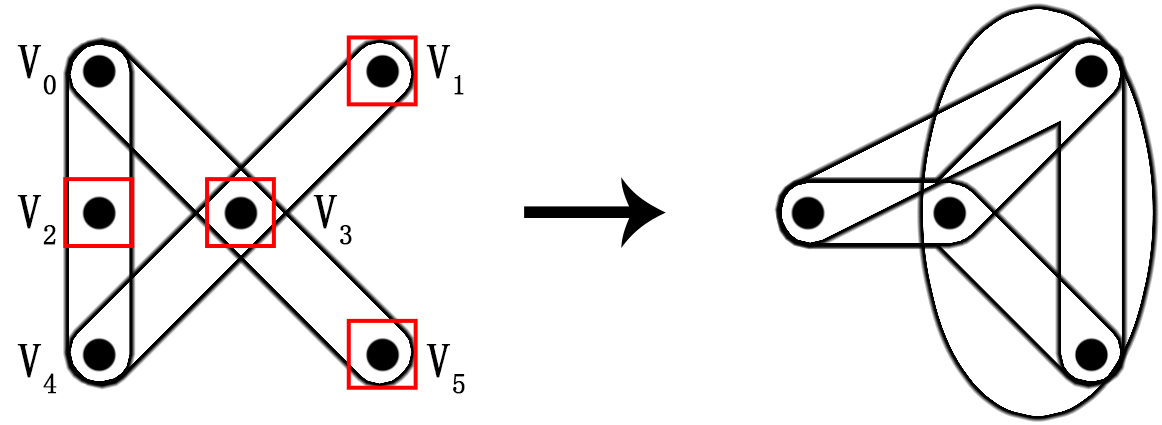}
\caption{}
\label{clique}
\end{figure}

\begin{proof}[Proof of Theorem~\ref{core14}]
Suppose the hypergraph $H_n^3$ satisfies $e(H_n^3) = cn^2$ for some $c > 0$.
We may assume that $H_n^3$ is linear and tripartite. We can make it tripartite by a random tri-partitioning as before and 
we can suppose that no vertex-pair is incident to 10 or more edges since it would give a $(12,10)$ configuration which is even 
stronger than what we want. Then the procedure which gives the linear hypergraph is the following: select an edge and remove the edges sharing two vertices with it and continue. In this way
we still have at least one tenth of the original edges and the remaining ones form a linear hypergraph.

By Theorem~\ref{63density} we know that there exist $\delta n^3$
subgraphs $F_6^3$ such that $e(F_6^3) = 3$ ($(6,3)$ configurations). In a tripartite linear hypergraph any such configuration looks like the graph on the left side of Figure~\ref{clique}. A degree one and a degree two vertex in each vertex class. To simplify our arguments below, let us partition all three vertex classes into two further classes. Every vertex selects its subclass independently at random with probability 1/2. The vertex classes now are $V_0, \ldots, V_6.$ At least $\delta' n^3$ $(6,3)$ configurations are distributed into the vertex classes in the same way as in Figure~\ref{clique} (with $\delta'\sim (1/2)^6\delta.$). Let us denote this family of $(6,3)$ configurations by ${\mathcal{F}}_n$.

There are triples of vertices of every $F_6^3$ that determine the configuration (almost) uniquely, no more than two $F_6^3$ configurations contain the same triple. 
This step of the proof is similar to the argument in Lemma \ref{dense74}. We will verify that any three vertices of an $F_6^3$ from vertex classes 
$V_2, V_3, V_1, V_5$ determines the configuration almost uniquely in ${\mathcal{F}}_n$. We analyze all four cases when $F\in{\mathcal{F}}_n$ has vertices 
\begin{enumerate}\label{cases}
\item $v_2\in V_2, v_3\in V_3$, and $v_1\in V_1.$ Form $v_1,v_3$ we recover $v_4$ as it is the unique third vertex of an edge of $F,$ Then $v_4,v_2$ gives $v_0$ and finally $v_0,v_3$ determines $v_5.$ The triple determines the configuration uniquely.
\item $v_2\in V_2, v_3\in V_3$  and $v_5\in V_5.$ This case is the same as the previous, we just flip the picture changing $v_1$ with $v_5.$ The triple determines the configuration uniquely.
\item $v_3\in V_3, v_1\in V_1$  and $v_5\in V_5.$ It is similar to the previous cases. From $v_1,v_3$ we recover $v_4.$ Using $v_5,v_3$ we get $v_0$ and from $v_0, v_4$ we get $v_2.$ The triple determines the configuration uniquely.
\item $v_2\in V_2, v_1\in V_1$  and $v_5\in V_5.$ These are the vertices which have degree one in the configuration. If $F_1,F_2\in{\mathcal{F}}_n$ share the three vertices then their union gives a core, a $(9,6)$ configuration. This would be a great news if we were looking for cores here, however our goal is different now. We can not exclude  this case. However if we had three, if $F_1,F_2, F_3\in{\mathcal{F}}_n$ sharing the three vertices then their union gives a $(12,9)$ system. Adding any edge incident to any vertex of $F_1\cup F_2\cup F_3$ in $H_n^3$ gives a $(14,10)$ configuration. So, we can suppose that the triple determines the configuration almost uniquely, with maximum multiplicity of two.
\end{enumerate}

We will construct an auxiliary four-partite three-uniform hypergraph 

$${\mathcal{H}}={\mathcal{H}}(V_2, V_3, V_1, V_5) $$ 

on vertex sets $V_2, V_3, V_1, V_5.$ 
For every configuration $F\in{\mathcal{F}}_n$ we add a clique, $K_4^3,$ to $\mathcal{H}$ on the vertices of $F$ in $V_2, V_3, V_1, V_5.$ 
The number of vertices of ${\mathcal{H}}$ is $|V_2|+|V_3|+|V_1|+|V_5|=O(n)$ and it is the union of $\delta'n^3$ edge disjoint cliques. 
No we can apply the Hypergraph Removal Lemma, Theorem \ref{rem}. There is a constant $\delta''>0$ which depends on $\delta$ only such that ${\mathcal{H}}$ contains at least $\delta'' n^4$ $K_4^3$ cliques. If $n$ is large enough then there is a clique such that the edges correspond to four different configurations $F_1,F_2,F_3,F_4\in{\mathcal{F}}_n,$ where $F_i$ determined by the vertices in case (i) above. For example if the clique has vertices $v_2, v_3, v_1, v_5$ then $F_3$ contains $v_3, v_1$ and $v_5$ but not $v_2.$ 

We claim that in this case $F_1\cup F_2\cup F_3\cup F_4$ has at most 14 vertices and at least 10 edges in in $H_n^3.$ We have the 4 vertices of the clique and every configuration has 3 vertices outside. $F_1$ and $F_3$ contain both $v_1$ and $v_3$ so they share a common vertex outside of the clique in $V_4.$ Similarly $F_2$ and $F_3$ contain both $v_3$ and $v_5$ so they have a common vertex in $V_0.$ The number of vertices then is at most $4+4\cdot3-2=14.$ The number of edges is $4\cdot3=12$ with multiplicity. From the previous argument we see that $F_1,F_3$ and $F_2,F_3$ both share an edge. These are the only cases when $F_i$ and $F_j$ share an edge when $i\neq j$. For any $i<j,$ $F_i$ and $F_j$ have exactly two common vertices in the clique. If they share an edge (other than what we considered already) then they share an edge which has at least two vertices outside of the clique. It can not be an edge incident to a vertex in $V_1.$ Two edges share a clique vertex there, one has multiplicity two and it has two vertices in the core so no more edge can overlap. The two remaining edges can not be the same since one contains a clique vertex in $V_1$ and the other is not.  The same holds for $V_5.$  Out of the four edges between $V_0$ and $V_4$ the edge of $F_3$ avoids the clique in $V_2,$ (otherwise it would contain all four vertices of the clique) and the other three share a clique vertex in $V_2,$ so $F_3$ can not share edge between $V_0$ and $V_4$ with any other configuration. The other three can not share an edge either since any pair of them has a common vertex in in the core outside of $V_2$ and any $F\in{\mathcal{F}}_n$ can be reconstructed uniquely from its edge in $V_0,V_2,V_4$ and one more vertex.

\begin{figure}[h!]
\centering
\includegraphics[width=0.6\linewidth]{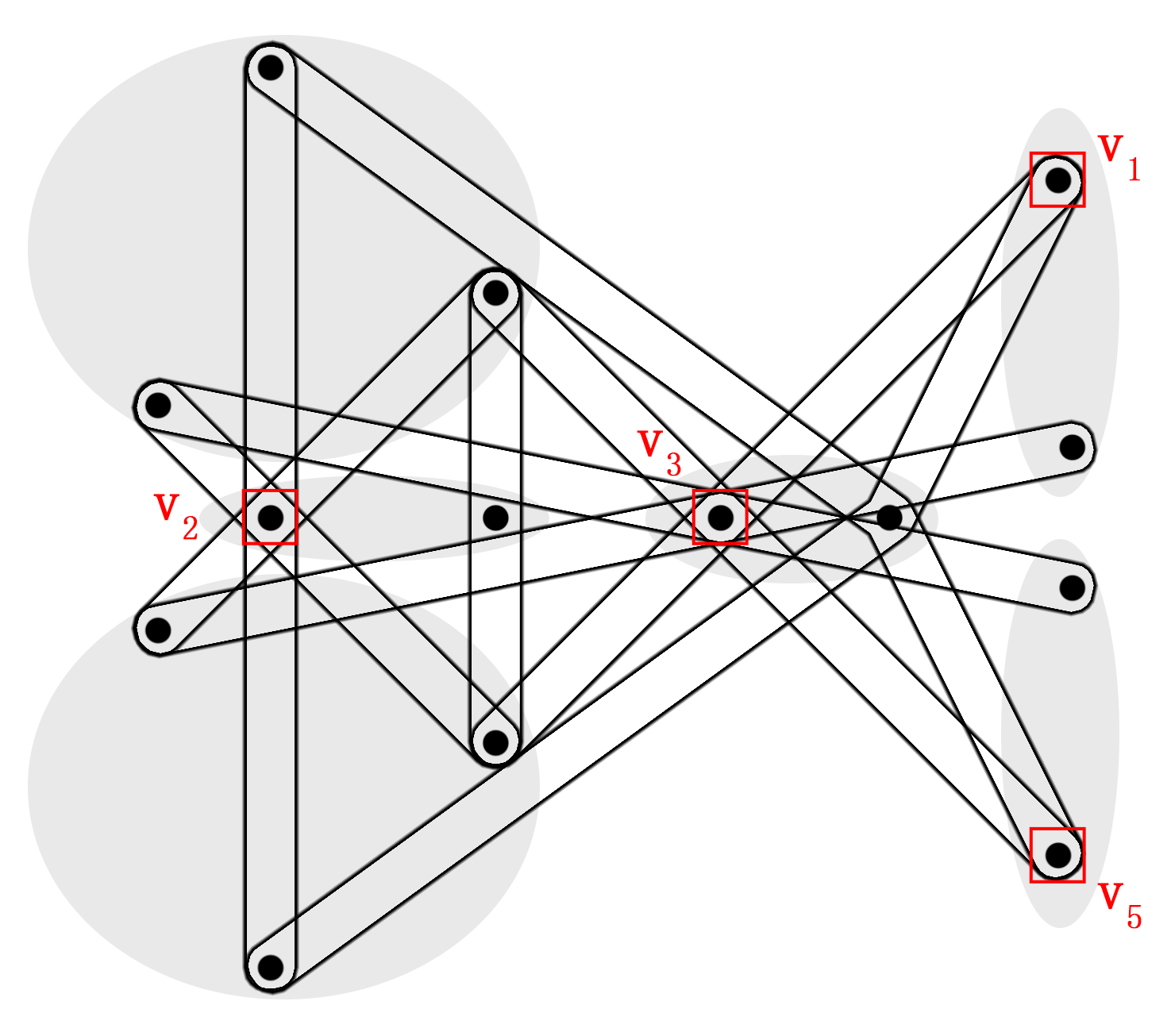}
\caption{}
\label{1410}
\end{figure}

Thus we get a structure of 10 edges on 14 vertices as seen in Figure~\ref{1410}.

\end{proof}

\subsection{$k=15$}

\medskip
The following is the first unconditional result with $o(n^2)$ edges.
\begin{thm}\label{core15}
$$\core(n,15)=o(n^2).$$
\end{thm}

\begin{proof}
We use the subgraphs of 6 vertices with 3 edges which are abundant by
Theorem~\ref{63density}. We will have $c'n^4$ pairs of these subgraphs
which overlap on two degree one vertices. These will be
subgraphs on 10 vertices with 6 edges, with two degree one vertices. We remove
one edge from each to create $c'n^4$ subgraphs on 9 vertices with 5 edges, with
three degree one vertices as seen
in Figure~\ref{two63}. Two of these will overlap on these degree one vertices,
giving us a core on 15 vertices.
\end{proof}

\begin{figure}[h!]
\centering
\includegraphics[width=0.25\linewidth]{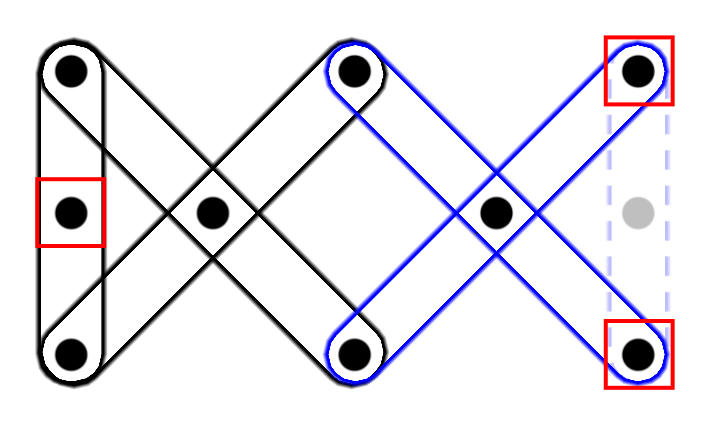}
\caption{}
\label{two63}
\end{figure}

\section{Large $k$ (16 to $\sqrt{n}$)}
In this range, when the number of edges of $H_n^3$ is at least $n^{3/2+c}$ for some $c>0,$ we show that one can guarantee a core of constant size, i.e. the size of the smallest core depends on $c$ but not on $n.$

Let us define a graph $G_m$ on $m=n^2$ vertices as follows. The vertices are the (ordered) pairs of the vertices $H_n^3.$ Two vertices, $(v_i,v_j)$ and $(v_s,v_t)$ are connected by an edge in $G_m$ if there is a vertex $v_r$ such that $(v_i,v_r,v_s)$ and $(v_j,v_r,v_t)$ are two distinct edges in $H_n^3.$ Note that the vertices are not necessarily distinct.
Jensen's inequality gives a lower bound on $e(G_m)$ in terms of the average degree in
$H_n^3$ and the number of vertices,
\[ e(G_m)\geq {{\frac{e(H_n^3)}{n}}\choose 2}n. \]
A cycle in $G_m$ is always a core in $H_n^3.$ It is yet another hard problem to determine the maximal girth of a graph in terms of the number of edges. We are going to use the asymptotically best known bounds for girth $g=2s+1$ see e.g.~in \cite{Bo} (page 1264) or in \cite{FS} . If $G_m$ has girth at least $g$ then $m\gtrsim e(G_m)^{s/s+1}.$

\begin{thm} Let $s>2$ be an integer. If $e(H_n^3)\gg n^{3/2+1/s}$ then $H_n^3$ contains a core on at most $3(2s+1)$ vertices.
$$\core(n,3(2s+1))=O(n^{3/2+1/s}).$$
\end{thm}

The proof follows from the observations above. We expect that much better bounds can be obtained and there is a chance to prove bounds without using the girth inequality in the auxiliary graph.

\section{Very large $k$ ($\sqrt{n}$ to $n$)}
For this range of $k$ we use a stripping algorithm similar to one commonly used to find cores in random constraint
satisfaction problems \cite{Mike}. 

Any hypergraph $H_n^3$ such that $e(H_n^3) \geq n-2$ contains a core which can be found by repeatedly removing
degree one vertices. For our purposes the following randomized step is convenient: Choose $n^{3/2} / \sqrt{e(H_n^3)}$ vertices uniformly at
random, and remove the rest. The probability of one edge being present
is $\big( ( n^{3/2} / \sqrt{e(H_n^3)} )/n \big)^3$, so the expected
number of edges among the chosen vertices is 
$\big( ( n^{3/2} / \sqrt{e(H_n^3)} )/n \big)^3 e(H_n^3)$, which is
equal to the number of vertices. Thus there exists a subgraph on that
many vertices with at least as many edges, so by the stripping process
above the existence of a core is guaranteed.

Thus $\core(n,k) \leq n^3 / k^2$ in this range.

\section{Summary} 
In the previous sections we considered the $\core(n,k)$ function. One can see more details examining the following variation of core. 

\[ \core^*(n,k) = \min \{ t : e(H_n^3) \geq t \Rightarrow \mbox{$H_n^3$ contains a core on $k$ vertices} \}, \]

The difference between the two notations is that here we are looking for the threshold where the existence of a core on exactly $k$ vertices is guaranteed. 

\medskip

\begin{table}[h]
\begin{tabular}{ |l|l|l| } 
\hline
\multicolumn{3}{ |c| }{About the $\core^*(n,k)$ function} \\
\hline
\multicolumn{3}{}{}\\
\hline
\multirow{3}{*}{very small $k$} & $k=4$ & $0.2857{n\choose 3}\leq \core^*(n,4)\leq  0.2871{n\choose 3}$ \\
 & $k=5$ & $\Omega(n^{5/2})\leq \core^*(n,5)= O(n^{5/2})$ \\ 
 & $k=6,7,8$ & $\Omega(n^2)\leq\core^*(n,k)=O(n^2)$ \\
\hline
\multirow{7}{*}{small $k$} 
 & $k=9$ & conj. $\core^*(n,9)=o(n^2)\implies BES(\ell=6)$ \\  
 & $k=10$ & $\core^*(n,10)=\Omega(n^2)$\\  
 & $k=11$ & conj. $\core^*(n,11)=o(n^2)\implies BES(\ell=8)$ \\  
 & $k=12$ & $\core^*(n,12)=o(n^2)$ for some triple systems\\  
 & $k=13$ & conj. $\core^*(n,13)=o(n^2)\implies BES(\ell=10)$ \\ 
 & $k=14$ & conj. $\core^*(n,14)=o(n^2)$ see Theorem \ref{core14} \\ 
 & $k=15$ & $\core^*(n,15)=o(n^2)$ see Theorem \ref{core15} \\
\hline
\end{tabular}
\medskip
\caption{{$BES(\ell=t)$ means that Conjecture \ref{BES} holds for $\ell=t.$} }
\label{table:mind}
\end{table}

\section{Acknowledgements}
The authors are thankful for the helpful discussions with Boris Bukh, Zolt\'an F\"uredi, Jie Han, and Dhruv Mubayi. We also thank the referees for the useful remarks. 


\end{document}